\documentclass[12pt]{article}
\usepackage{color}
\usepackage{graphicx}
\usepackage{colortbl}
\usepackage{array}
\usepackage{amssymb}
\usepackage{amsmath}
\usepackage{mathrsfs}
\usepackage{dcolumn}
\usepackage{longtable}
\usepackage{hhline}
\usepackage{graphics}
\usepackage{amssymb}
\usepackage{amsmath}
\usepackage{amsthm}
\usepackage{mathrsfs}
\usepackage{amsfonts}
\usepackage{latexsym}
\usepackage{color}
\usepackage{epsf}

\newtheorem{thrm}{Theorem}[section]

\newtheorem{corr}[thrm]{Corollary}

\newtheorem{lemm}[thrm]{Lemma}

\title{\textbf{Some Formulae of Genocchi Polynomials\\ of Higher Order}}
\author{\textbf{Cristina B. Corcino}\\{\large\bf Roberto B. Corcino}\\{Research Institute for Computational}\\ {Mathematics and Physics}\\Cebu Normal University\\Cebu City, Philippines \vspace{9pt}\\   {\large\bf Joy Ann A. Ca\~nete} \\ Mathematics Department \\ Visayas State University\\ Baybay City, Philippines}

\begin{document}
\maketitle
\begin{abstract}
In this paper, some formulae for Genoochi polynomials of higher order are derived using the fact that sets of Bernoulli and Euler polynomials of higher order form basis for the polynomial space. 
\end{abstract}

\section{Introduction}
A variety of polynomials, their extensions, and variants, have been extensively investigated, mainly due to their potential applications in diverse research areas. Motivated by their importance and potential for applications in a variety of research fields, numerous polynomials and their extensions have recently been introduced and investigated. One of these polynomials is the Genocchi polynomials that have been extensively studied in many different context in such branches of mathematics, for instance, in elementary number theory, complex analytic number theory, calculus and many more. One application of this polynomial is to study a matrix formulation and formulas obtained are used as check formulas of similar formula in \cite{ref4}. One paper of S. Araci in \cite{ref5} deals with the applications of umbral calculus on fermionic $p$-adic integral on $\mathbb{Z}_p$ and from these, S. Araci derived some new identities on Genocchi numbers and polynomials. Another paper that presents new numerical method for solving fractional differential equations was based on Genocchi polynomials operational matrix through collocation method in the paper of A. Isah and C. Phang (see \cite{ref8}), these polynomial properties are utilized to reduce the given problems to a system of algebraic equations. This also helps other researchers to do further researches for Genocchi polynomials and Genocchi polynomials of order $k$ for it has limited literature.

\smallskip
The history of Genocchi numbers can be tracked back to Italian mathematician Angelo Genocchi (1817-1889). From Genocchi to the present time, Genocchi numbers have been extensively studied in many different contexts in mathematics . Many studies and literature provide relations of Genocchi numbers to Bernoulli and Euler numbers, especially Euler numbers. Bernoulli, Euler and Genocchi numbers defined by exponential generating function (see \cite{ref1, ref3, ref4})
\begin{align}
\sum_{n=0}^{\infty}B_n\frac{t^n}{n!}&=\frac{t}{e^t-1},\;\;\; |t|<2\pi\label{eqn01}\\
\sum_{n=0}^{\infty}E_n\frac{t^n}{n!}&=\frac{2}{e^t+1},\;\;\; |t|<\pi\label{eqn02}\\
\sum_{n=0}^{\infty}G_n\frac{t^n}{n!}&=\frac{2t}{e^t+1},\;\;\; |t|<\pi\label{eqn03}.
\end{align}
The Bernoulli, Euler and Genocchi polynomials are defined via generating functions to be, respectively,
\begin{align}
\sum_{n=0}^{\infty}B_n(x)\frac{t^n}{n!}&=\frac{t}{e^t-1}e^{xt},\;\;\; |t|<2\pi\label{eqn001}\\
\sum_{n=0}^{\infty}E_n(x)\frac{t^n}{n!}&=\frac{2}{e^t+1}e^{xt},\;\;\; |t|<\pi\label{eqn002}\\
\sum_{n=0}^{\infty}G_n(x)\frac{t^n}{n!}&=\frac{2t}{e^t+1}e^{xt},\;\;\; |t|<\pi\label{eqn003},
\end{align}
where, when $x = 0$, $B_n(0) = B_n$, $E_n(0) = E_n$ and $G_n(0) = G_n$. (see \cite{ref3, ref4, ref7, ref9}). Bernoulli and Euler polynomials have been extensively studied by various researchers, one specific paper from D.S. Kim and et al. \cite{ref9}, that deals with some formulae of the product of two Bernoulli and Euler polynomials. This was then extended by Araci et al. \cite{ref4} to Genocchi polynomial. 

\smallskip
It is known that the Bernoulli and Euler Polynomials of order $k$ are defined respectively by the generating function
\begin{align}
\sum_{n=0}^{\infty}B_n^{k}(x)\frac{t^n}{n!}&=\left(\frac{t}{e^t-1}\right)^ke^{xt},\;\;\; |t|<2\pi\label{eqn0001}\\
\sum_{n=0}^{\infty}E_n^{k}(x)\frac{t^n}{n!}&=\left(\frac{2}{e^t+1}\right)^ke^{xt},\;\;\; |t|<\pi.\label{eqn0002}
\end{align}
In the special case, $x = 0$, $B_n^k(0) = B_n^k$ and $E_n^k(0) = E_n^k$ are called the 
Bernoulli and Euler numbers of order $k$, respectively. From (\ref{eqn0001}) and (\ref{eqn0002}), we have $B_n^0(x) = x^n$ and $E_n^0(x) = x^n$. It is then not difficult to show that (see \cite{ref10, ref12} ), for Bernoulli,
$$\frac{d}{dx}B_n^k(x)=nB_{n-1}^k(x)$$
and
$$B_{n}^k(x+1)-B_{n}^k(x)=nB_{n-1}^{k-1}(x).$$
For Euler,
$$\frac{d}{dx}E_n^k(x)=nE_{n-1}^k(x)$$
and
$$E_{n}^k(x+1)+E_{n}^k(x)=2E_{n-1}^{k-1}(x).$$
In the paper of D.S. Kim and T. Kim \cite{ref10, ref12}, they introduced
$$\mathcal{P}_n=\left\{p(x)\in\mathbb{Q}[x]:\deg p(x)\leq n\right\}$$
to be the $(n +1)$-dimensional vector space over $\mathbb{Q}$. Probably, $\{1,x,\ldots, x_n\}$ is the most natural basis for $\mathcal{P}_n$ but
$$\left\{B_{0}^k(x), B_{1}^k(x), \ldots, B_{n}^k(x)\right\} \;\;\mbox{and}\;\;\left\{E_{0}^k(x), E_{1}^k(x), \ldots, E_{n}^k(x)\right\}$$
are also good bases for the space $\mathcal{P}_n$ , (see \cite{ref10, ref12}. Then $p(x)$ can be expressed by
$$p(x)=\sum_{j=0}^na_jB_{j}^k(x)$$
and
$$p(x)=\sum_{j=0}^nb_jE_{j}^k(x).$$
From this, D.S Kim and T. Kim develop some interesting identities of Bernoulli and Euler Polynomials, (see \cite{ref10, ref12}.

\smallskip
Araci et al. \cite{ref3} introduced the Genocchi polynomial of higher order defined by the relation,
\begin{equation}\label{eqn1}
\sum_{n=0}^{\infty}G_n^{k}(x)\frac{t^n}{n!}=\left(\frac{2t}{e^t+1}\right)^ke^{xt}, \;\;\;k\in\mathbb{Z}^+, \;\;\;|t|<\pi.
\end{equation}
Note that when $k = 1$, (\ref{eqn1}) reduces to (\ref{eqn003}), hence $G^{(1)}_n(x) = G_n(x)$. In a separate paper, Araci et al. \cite{ref4} established new formulae for Genocchi polynomials including Euler polynomials and Bernoulli polynomials.  These properties were derived from the basis of Genocchi which leads Araci et al. \cite{ref3} to obtain new identities and extended it to Bernoulli and Euler polynomials.

\smallskip
In this paper, new formulae of Genocchi polynomial of order $k$ are established, which are parallel to those in \cite{ref4}. Furthermore, derived results are presented in terms of Bernoulli and Euler polynomials especially the latter polynomial since Genocchi polynomials are closely related to it.

\section{Preliminary Results}
In this section, some identities of Genocchi polynomials of higher order are obtained, which are parallel to those of Bernoulli and Euler polynomials of higher order. Also, relationship between Euler numbers of higher order and Genocchi numbers of higher order is obtained.

\smallskip
When $x = 0$, \eqref{eqn1} becomes
\begin{equation}\label{eqn2}
\sum_{n=0}^{\infty}G_n^{k}\frac{t^n}{n!}=\left(\frac{2t}{e^t+1}\right)^k,
\end{equation}
where $G^k_n$ are called higher order Genocchi numbers. When $k = 1$, \eqref{eqn2} reduces to \eqref{eqn03}.

\smallskip
Now, differentiating both sides of \eqref{eqn1} gives
\begin{align*}
\sum_{n=0}^{\infty}\frac{d}{dx}G_n^k(x)\frac{t^n}{n!}&=\left(\frac{2t}{e^t+1}\right)^kte^{xt}\\
&=\sum_{n=0}^{\infty}G_n^{k}\frac{t^{n+1}}{n!}=\sum_{n=0}^{\infty}nG_{n-1}^{k}(x)\frac{t^{n}}{n!}.
\end{align*}
Hence,
\begin{equation}\label{eqn3}
\frac{d}{dx}G_n^k(x)=nG_{n-1}^{k}(x).
\end{equation}
Since the differentiation of Genocchi polynomial of higher order exists, it follows that it also has integration, since it is commonly understood that differentiation and integration have an inverse relationship. So from (\ref{eqn3}), 
$$\int\frac{d}{dx}G_n^k(x)=\int nG_{n-1}^{k}(x),$$
which gives
\begin{align*}
G_n^k(x)&=n\int G_{n-1}^{k}(x)\\
\frac{G_n^k(x)}{n}&=\int G_{n-1}^{k}(x).
\end{align*}
Thus,
$$\int G_{n}^{k}(x)=\frac{G_{n+1}^k(x)}{n+1}.$$
That is,
\begin{equation}\label{eqn4}
\int_a^b G_{n}^{k}(x)=\frac{G_{n+1}^k(b)-G_{n+1}^k(a)}{n+1}.
\end{equation}
The following lemma contains an expression of $G_{n}^{k}(x)$ as polynomial in $x$ with  Genocchi numbers of higher order as the coefficients.
\begin{lemm}
The Genocchi polynomials of higher order satisfy the following relation
\begin{equation}\label{eqn5}
G_{n}^{k}(x)=\sum_{m=0}^n\binom{n}{m}G_{n}^kx^{n-m}.
\end{equation}
\end{lemm}
\begin{proof}
Using equation \eqref{eqn1}, 
\begin{align*}
\sum_{n=0}^{\infty}G_n^{k}(x)\frac{t^n}{n!}&=\sum_{n=0}^{\infty}G_n^{k}\frac{t^n}{n!}\sum_{n=0}^{\infty}\frac{(xt)^n}{n!}\\
&=\sum_{n=0}^{\infty}\left\{\sum_{m=0}^{n}\binom{n}{m}G_{n-m}^{k}x^{n-m}\right\}\frac{t^n}{n!}.
\end{align*}
Comparing the coefficients of $\frac{t^n}{n!}$ completes the proof of the lemma.
\end{proof}

\begin{lemm}\label{lem1}
The following relation holds
\begin{equation}\label{eqn6}
G^k_n(x) = (-1)^{n+k}G^k_n(k - x)
\end{equation}
such that, when $x=1$,
\begin{equation}\label{eqn7}
G^k_n(1) = (-1)^{n+k}G^k_n(k - 1).
\end{equation}
\end{lemm}
\begin{proof}
By replacing $x$ with $k-x$, equation \eqref{eqn1} gives
$$\sum_{n=0}^{\infty}G_n^{k}(k-x)\frac{t^n}{n!}=\left(\frac{2t}{e^t+1}\right)^ke^{kt}e^{-xt}.$$
Replacing $t$ with $-t$ yields
\begin{align*}
\sum_{n=0}^{\infty}(-1)^nG_n^{k}(k-x)\frac{t^n}{n!}&=\left(\frac{-2t}{e^{-t}+1}\right)^ke^{-kt}e^{xt}\\
&=(-1)^k\left(\frac{2t}{e^{t}+1}\right)^ke^{xt}\\
&=(-1)^k\sum_{n=0}^{\infty}G_n^{k}(x)\frac{t^n}{n!}.
\end{align*}
Comparing the coefficients of $\frac{t^n}{n!}$ gives \eqref{eqn6}.
\end{proof}

\smallskip
Kim, et al. \cite{ref9} introduced the set 
$$\{E_0(x), E_1(x), \ldots , E_n(x)\} \;\;\mbox{and}\;\; \{B_0(x), B_1(x), \ldots , B_n(x)\}$$ 
as Euler and Bernoulli basis for the space of polynomials of degree less than or equal to $n$ with coefficients in $\mathbb{Q}$, respectively. Since Bernoulli and Euler polynomials are closely related to Genocchi polynomials, the work in \cite{ref9} was then extended to Genocchi polynomial by $S$.
Araci, et al. \cite{ref4} who introduced 
$$\mathcal{P}_n = \{p(x) \in \mathbb{Q}[x] : \deg p(x) \leq n\}$$ 
to be the $(n + 1)$-dimensional vector space over $\mathbb{Q}$ and obtained 
$$\{G_1(x), G_2(x), \ldots , G_n(x), G_{n+1}(x)\}$$ 
as a good basis for space $\mathcal{P}_n$. From this, 
$$\left\{G_k^k(x), G_{k+1}^k(x), \ldots , G_{n+k-1}^k(x), G_{n+k}^k(x)\right\}$$ 
is also a good basis for the space $\mathcal{P}_n$ as stated in the work of Araci, et al. \cite{ref3}.

\smallskip
If $p(x) \in \mathcal{P}_n$, then $p(x)$ can be expressed as a linear combination of the elements in $\left\{G^k_l\right\}^n_{l=0}$. That is,
\begin{equation}\label{eqn9}
p(x)=\sum_{l=k}^{n+k}a_lG_{l}^k(x).
\end{equation}
Let us introduce the polynomial
\begin{equation}\label{eqn10}
p(x)=\sum_{l=k}^{n+k}G_{l}^k(x)x^{n+k-l}.
\end{equation}
with $n \in \mathbb{N}$, then taking the first derivative of (\ref{eqn10})
\begin{align*}
p^{\prime}(x)&=\sum_{l=k}^{n+k}\{G_{l}^k(x)(n+k-l)x^{n+k-l-1}+lG_{l-1}^k(x)x^{n+k-l}\}\\
&=\sum_{l=k}^{n+k-1}(n+k-l)G_{l}^k(x)x^{n+k-l-1}+\sum_{l=k}^{n+k}lG_{l-1}^k(x)x^{n-l}\\
&=\sum_{l=k+1}^{n+k}(n+k-(l-1))G_{l-1}^k(x)x^{n+k-(l-1)-1}+\sum_{l=k+1}^{n+k}lG_{l-1}^k(x)x^{n+k-l}\\
&\;\;\;\;+kG_{k-1}^k(x)x^{n-k}.
\end{align*}
Since $G_{k-1}^k(x)=0$, 
$$p^{\prime}(x)=(n+k+1)\sum_{l=k+1}^{n+k}G_{l-1}^k(x)x^{n+k-l}.$$
For the second derivative, one can easily verify that
$$p^{\prime\prime}(x)=(n+k+1)(n+k)\sum_{l=k+2}^{n+k}G_{l-2}^k(x)x^{n+k-l}.$$
Continuing this process yields
$$p^{(j)}(x)=(n+k+1)(n+k)\ldots (n+k+2-j)\sum_{l=k+j}^{n+k}G_{l-j}^k(x)x^{n+k-l}.$$
The following lemma states formally this result.
\begin{lemm}\label{lem2}
The $j$th derivative of the polynomial $p(x)$ in \eqref{eqn10} is given by
$$p^{(j)}(x)=\frac{(n+k+1)!}{(n+k+1-j)!}\sum_{l=k+j}^{n+k}G_{l-j}^k(x)x^{n+k-l}.$$
\end{lemm}

\section{Main Results}
In the complex plane, wequation (\ref{eqn1}) can be written into the following:
\begin{align*}
\sum_{n=k}^{\infty}G_n^{k}(z)\frac{t^n}{n!}&=\left(\frac{2}{e^t+1}\right)^ke^{zt}t^k=\sum_{n=0}^{\infty}E_n^{k}(z)\frac{t^{n+k}}{n!}\\
&=\sum_{n=k}^{\infty}E_{n-k}^{k}(z)\frac{t^{n}}{(n-k)!}\\
&=\sum_{n=k}^{\infty}\frac{n!E_{n-k}^{k}(z)}{(n-k)!}\frac{t^{n}}{n!}.
\end{align*}
Comparing the coefficients of $\frac{t^{n}}{n!}$ yields
\begin{equation}\label{eqn00}
G_n^{k}(z)=\frac{n!}{(n-k)!}E_{n-k}^{k}(z).
\end{equation}
Consequently, this gives
\begin{align*}
G_{n+k}^{k}(z)&=\frac{(n+k)!}{((n+k)-k)!}E_{(n+k)-k}^{k}(z)\\
&=\frac{(n+k)!}{n!}E_{n}^{k}(z)\\
&=(n + k)(n + k - 1) · · · (n + 1)E_{n}^{k}(z)\\
&=(n + k)_kE_{n}^{k}(z),
\end{align*}
where $(n + k)_k = (n + k)(n + k -1)\ldots (n + 1)$ is called the falling factorial of $n+k$ of degree $k$. Note that, when $z=0$, 
\begin{equation}\label{eqn011}
\frac{G_{n+k}^{k}(z)}{(n+k)_k}=E_{n}^k,
\end{equation}
and when $z=k-1$, 
\begin{equation}\label{eqn012}
\frac{G_{n+k}^{k}(k-1)}{(n+k)_k}=E_{n}^k(k-1).
\end{equation}
\begin{thrm}\label{thm1}
The following identity holds
\begin{align*}
\sum_{l=k}^{n+k}G_{l}^k(x)x^{n+k-l}=\sum_{j=0}^n\frac{\binom{n+k+1}{j}}{n+k-j+2}\left(\sum_{l=k+j-1}^{n+k}(-1)^{k+l-j+1}G_{l-j+1}^k(k-1) - G_{n+k-j+1}^k\right)B_j^k(x),
\end{align*}
where $B_j^k(x)$ denotes Bernoulli polynomial of order $k$.
\end{thrm}
\begin{proof}
On account of the properties of Bernoulli basis for the space polynomials of degree less than or equal to $n$ with coefficients in $\mathbb{Q}$, then $p(x)$ in \eqref{eqn10} can be written as follows: (see \cite{ref10})
\begin{equation}\label{eqn11}
p(x)=\sum_{j=0}^na_jB_j^k(x).
\end{equation}
Taking the first and second derivative
\begin{align*}
p^{\prime}(x)&=\sum_{j=1}^na_jjB_{j-1}^k(x)=\sum_{j=1}^na_j\frac{j!}{(j-1)!}B_{j-1}^k(x),\\
p^{\prime\prime}(x)&=\sum_{j=2}^na_jj(j-1)B_{j-2}^k(x)=\sum_{j=2}^na_j\frac{j!}{(j-2)!}B_{j-2}^k(x).
\end{align*}
Continuing in this manner yields
\begin{align*}
p^{(n-1)}(x)&=\sum_{j=n-1}^na_j\frac{j!}{(j-n+1)!}B_{j-n+1}^k(x)\\
&=a_{n-1}\frac{(n-1)!}{((n-1)-n+1)!}B_{(n-1)-n+1}^k(x)+a_{n}\frac{n!}{(n-n+1)!}B_{n-n+1}^k(x)\\
&=a_{n-1}(n-1)!B_{0}^k(x)+a_{n}n!B_{1}^k(x).
\end{align*}
Replacing $n$ with $j$ gives
$$p^{(j-1)}(x)=a_{j-1}(j-1)!B_{0}^k(x)+a_{j}j!B_{1}^k(x).$$
Hence, with $B_0^k(x)=1$ and $B_1^k (x) = x-\frac{k}{2}$, 
\begin{align*}
p^{(j-1)}(1)-p^{(j-1)}(0)&=a_{j-1}(j-1)!B_{0}^k(1)+a_{j}j!B_{1}^k(1)-[a_{j-1}(j-1)!B_{0}^k(0)+a_{j}j!B_{1}^k(0)]\\
&=a_jj!\left(1-\frac{k}{2}\right)-a_jj!\left(0-\frac{k}{2}\right)=a_jj!.
\end{align*}
Thus, 
\begin{equation}\label{eqnn1}
a_j=\frac{1}{j!}[p^{(j-1)}(1)-p^{(j-1)}(0)].
\end{equation}
Using Lemma \ref{lem2} and \eqref{eqn7}, 
\begin{align*}
a_j&=\frac{1}{j!}\left(\frac{(n+k+1)!}{(n+k+1-j+1)!}\sum_{l=k+j-1}^{n+k}G_{l-j+1}^k(1)1^{n+k-l}\right.\\
&\;\;\;\;\;\;\;\;\;\left.-\frac{(n+k+1)!}{(n+k+1-j+1)!}\sum_{l=k+j-1}^{n+k}G_{l-j+1}^k(0)0^{n+k-l}\right)\\
&=\frac{1}{j!}\frac{(n+k+1)!}{(n+k-j+2)!}\left(\sum_{l=k+j-1}^{n+k}G_{l-j+1}^k(1) - G_{n+k-j+1}^k(0)\right)\\
&=\frac{\binom{n+k+1}{j}}{n+k-j+2}\left(\sum_{l=k+j-1}^{n+k}(-1)^{k+l-j+1}G_{l-j+1}^k(k-1) - G_{n+k-j+1}^k\right).\\
\end{align*}
Thus, \eqref{eqn11} gives
$$p(x)=\sum_{j=0}^n\frac{\binom{n+k+1}{j}}{n+k-j+2}\left(\sum_{l=k+j-1}^{n+k}(-1)^{k+l-j+1}G_{l-j+1}^k(k-1) - G_{n+k-j+1}^k\right)B_j^k(x),$$
which is exactly the desired identity. 
\end{proof}

\begin{corr}
The following equality holds:
\begin{align*}
\sum_{l=k}^{n+k}G_{l}^k(x)x^{n+k-l}&=\sum_{j=0}^n\frac{\binom{n+k+1}{j}}{n+k-j+2}\Bigl(\sum_{l=k+j-1}^{n+k}\left[(-1)^{k+l-j+1}(l-j+1)_kE_{l-j+1-k}^k(k-1) \right.\Bigr.\\
&\;\;\;\;\;\;\;\Bigl.\left. -(n+k-j+1)_kE_{n-j+1}^k\right]\Bigr)B_j^k(x),
\end{align*}
\end{corr}
\begin{proof}
Applying equations \eqref{eqn011} and \eqref{eqn012} immediately proves the corollary.
\end{proof}

\smallskip
The identity in the next theorem is obtained by using the fact that the set of Euler polynomials of higher order forms a basis for the space of polynomials $\mathcal{P}_n$.
\begin{thrm}\label{thm2}
The following identity holds
\begin{align*}
\sum_{l=k}^{n+k}G_{l}^k(x)x^{n+k-l}=\sum_{j=0}^n\frac{1}{2}\binom{n+k+1}{j}\left(\sum_{l=k+j}^{n+k}(-1)^{k+l-j}G_{l-j}^k(k-1) + G_{n+k-j}^k\right)E_j^k(x),
\end{align*}
where $E_j^k(x)$ denotes Euler polynomial of order $k$.
\end{thrm}
\begin{proof}
Expressing the polynomial $p(x)$ as a linear combination of Euler polynomials of higher order yields
\begin{equation}\label{eqn12}
p(x)=\sum_{j=0}^nb_jE_j^k(x).
\end{equation}
Taking the first and second derivative
\begin{align*}
p^{\prime}(x)&=\sum_{j=1}^nb_jjE_{j-1}^k(x)=\sum_{j=1}^nb_j\frac{j!}{(j-1)!}E_{j-1}^k(x)\\
p^{\prime\prime}(x)&=\sum_{j=2}^nb_jj(j-1)E_{j-2}^k(x)=\sum_{j=2}^nb_j\frac{j!}{(j-2)!}E_{j-2}^k(x).
\end{align*}
Continuing in this manner yields
\begin{align*}
p^{(n-1)}(x)&=\sum_{j=n-1}^nb_j\frac{j!}{(j-n+1)!}E_{j-n+1}^k(x)\\
p^{(n)}(x)&=\sum_{j=n}^nb_j\frac{j!}{(j-n)!}E_{j-n}^k(x)\\
&=b_nn!E_{0}^k(x).
\end{align*}
Replacing $n$ with $j$ gives
$$p^{(j)}(x)=b_{j}j!E_{0}^k(x).$$
Hence, with $E_0^k(x)=1$, 
\begin{align*}
p^{(j)}(1)+p^{(j)}(0)&=b_{j}j![E_{0}^k(1)+E_{0}^k(0)]=b_jj![1+1]=2b_jj!.
\end{align*}
Thus, 
\begin{equation}\label{eqnn2}
b_j=\frac{1}{2j!}[p^{(j)}(1)+p^{(j)}(0)].
\end{equation}
Using Lemma \ref{lem2} and \eqref{eqn7}, 
\begin{align*}
b_j&=\frac{1}{2j!}\left(\frac{(n+k+1)!}{(n+k+1-j)!}\sum_{l=k+j}^{n+k}G_{l-j}^k(1)1^{n+k-l}\right.\\
&\;\;\;\;\;\;\;\;\;\left.+\frac{(n+k+1)!}{(n+k+1-j)!}\sum_{l=k+j}^{n+k}G_{l-j}^k(0)0^{n+k-l}\right)\\
&=\frac{1}{2j!}\frac{(n+k+1)!}{(n+k-j+1)!}\left(\sum_{l=k+j}^{n+k}G_{l-j}^k(1) + G_{n+k-j}^k(0)\right)\\
&=\frac{1}{2}\binom{n+k+1}{j}\left(\sum_{l=k+j}^{n+k}(-1)^{k+l-j}G_{l-j}^k(k-1) + G_{n+k-j}^k\right).\\
\end{align*}
Thus, \eqref{eqn11} gives
$$p(x)=\sum_{j=0}^n\frac{1}{2}\binom{n+k+1}{j}\left(\sum_{l=k+j}^{n+k}(-1)^{k+l-j}G_{l-j}^k(k-1) + G_{n+k-j}^k\right)E_j^k(x),$$
which is exactly the desired identity. 
\end{proof}

\smallskip
The following corollary immediately follows from Theorem \ref{thm2} using equations \eqref{eqn011} and \eqref{eqn012}.
\begin{corr}
The following equality holds:
\begin{align*}
\sum_{l=k}^{n+k}G_{l}^k(x)x^{n+k-l}&=\sum_{j=0}^n\frac{1}{2}\binom{n+k+1}{j}\Bigl(\sum_{l=k+j-1}^{n+k}\left[(-1)^{k+l-j+1}(l-j)_kE_{l-j-k}^k(k-1) \right.\Bigr.\\
&\;\;\;\;\;\;\;\Bigl.\left. +(n+k-j)_kE_{n-j}^k\right]\Bigr)E_j^k(x).
\end{align*}
\end{corr}

\smallskip
Consider the following form of polynomial
\begin{equation}\label{eqn13}
p(x)=\sum_{l=k}^{n+k}\frac{1}{l!(n+k-l)!}G_{l}^k(x)x^{n+k-l}.
\end{equation}
Then, applying the first derivative gives
\begin{align*}
p^{\prime}(x)&=\sum_{l=k}^{n+k}\frac{1}{l!(n+k-l)!}\left[G_{l}^k(x)(n+k-l)x^{n+k-l-1}+lG_{l-1}^k(x)x^{n+k-l}\right]\\
&=\sum_{l=k}^{n+k-1}\frac{1}{l!(n+k-l)!}G_{l}^k(x)(n+k-l)x^{n+k-l-1}\\
&\;\;\;\;\;\;\;+\sum_{l=k}^{n+k}\frac{1}{l!(n+k-l)!}lG_{l-1}^k(x)x^{n+k-l}\\
&=\sum_{l=k+1}^{n+k}\frac{1}{(l-1)!(n+k-l+1)!}G_{l-1}^k(x)(n+k-l+1)x^{n+k-l}\\
&\;\;\;\;\;\;\;+\sum_{l=k+1}^{n+k}\frac{1}{l!(n+k-l)!}lG_{l-1}^k(x)x^{n+k-l}\\
&=\sum_{l=k+1}^{n+k}\frac{1}{(l-1)!(n+k-l)!}G_{l-1}^k(x)x^{n+k-l}\\
&\;\;\;\;\;\;\;+\sum_{l=k+1}^{n+k}\frac{1}{(l-1)!(n+k-l)!}G_{l-1}^k(x)x^{n+k-l}\\
&=2\sum_{l=k+1}^{n+k}\frac{1}{(l-1)!(n+k-l)!}G_{l-1}^k(x)x^{n+k-l}.
\end{align*}
Applying the same process gives the following second and third derivative of the polynomial
\begin{align*}
p^{\prime\prime}(x)&=2^2\sum_{l=k+2}^{n+k}\frac{1}{(l-2)!(n+k-l)!}G_{l-2}^k(x)x^{n+k-l},\\
p^{\prime\prime\prime}(x)&=2^3\sum_{l=k+3}^{n+k}\frac{1}{(l-3)!(n+k-l)!}G_{l-3}^k(x)x^{n+k-l}.
\end{align*}
Thus, by induction, the following lemma is proved.
\begin{lemm}\label{lem111}
The $j$th derivative of polynomial $p(x)$ in \eqref{eqn13} is given by
\begin{equation}\label{eqn14}
p^{(j)}(x)=2^j\sum_{l=k+j}^{n+k}\frac{1}{(l-j)!(n+k-l)!}G_{l-j}^k(x)x^{n+k-l}.
\end{equation}
\end{lemm}
\begin{thrm}\label{thm3}
The following relation holds:
\begin{align*}
\sum_{l=k}^{n+k}\frac{1}{l!(n+k-l)!}G_{l}^k(x)x^{n+k-l}=\sum_{j=k}^{n+k}\frac{2^{j-k-1}}{j!}\left(\sum_{l=j}^{n+k}\frac{(-1)^{l-j}G_{l+k-j}^k(k-1)}{(l+k-j)!(n+k-l)!}+\frac{G_{n+2k-j}^k}{(n+2k-j)!}\right).
\end{align*}
\end{thrm}
\begin{proof}
Using the fact that the set of Genocchi polynomials of higher order forms a basis for the space of polynomials $\mathcal{P}_n$, 
\begin{equation}\label{eqn15}
p(x)=\sum_{l=k}^{n+k}c_lG_{l}^k(x).
\end{equation}
Applying the first derivative yields
\begin{align*}
p^{\prime}(x)&=\sum_{l=k}^{n+k}c_llG_{l-1}^k(x)=\sum_{l=k+1}^{n+k}c_l\frac{l!}{(l-1)!}G_{l-1}^k(x),\\
p^{\prime\prime}(x)&=\sum_{l=k+1}^{n+k}c_l\frac{l!}{(l-1)!}(l-1)G_{l-2}^k(x)=\sum_{l=k+2}^{n+k}c_l\frac{l!}{(l-2)!}G_{l-2}^k(x),\\
p^{\prime\prime\prime}(x)&=\sum_{l=k+2}^{n+k}c_l\frac{l!}{(l-2)!}(l-2)G_{l-3}^k(x)=\sum_{l=k+3}^{n+k}c_l\frac{l!}{(l-3)!}G_{l-3}^k(x),\\
&\vdots\\
 p^{(n)}(x)&=\sum_{l=k+n}^{n+k}c_l\frac{l!}{(l-n)!}G_{l-n}^k(x)\\
 &=c_{n+k}\frac{(n+k)!}{((n+k)-n)!}G_{(n+k)-n}^k(x)\\
 &=c_{n+k}\frac{(n+k)!}{k!}G_{k}^k(x).
 \end{align*}
From equation \eqref{eqn00}, 
$$G_{k}^k(x)=k!E_0^k(x)\;\;\mbox{and}\;\;G_{k+1}^k(x)=(k+1)!E_1^k(x).$$
Thus,
$$p^{(n)}(x)=c_{n+k}\frac{(n+k)!}{k!}k!E_{0}^k(x)=c_{n+k}(n+k)!E_{0}^k(x).$$
Replacing $n$ with $j-k$ yields
$$p^{(j-k)}(x)=c_{j}j!E_{0}^k(x).$$
Then, we have
\begin{align*}
p^{(j-k)}(0)+p^{(j-k)}(1)&=c_{j}j![E_{0}^k(0)+E_{0}^k(1)]\\
&=c_{j}j![1+1]=2c_{j}j!
\end{align*}
Thus, 
\begin{align*}
c_j&=\frac{1}{2j!}[p^{(j-k)}(1)+p^{(j-k)}(0)]\\
&=\frac{1}{2j!}\left(2^{j-k}\sum_{l=j}^{n+k}\frac{1}{(l-(j-k))!(n+k-l)!}G_{l-(j-k)}^k(1)1^{n+k-l}\right.\\
&\;\;\;\;\;\;\;\left.+2^{j-k}\sum_{l=j}^{n+k}\frac{1}{(l-(j-k))!(n+k-l)!}G_{l-(j-k)}^k(0)0^{n+k-l}\right)\\
&=\frac{1}{2j!}\left(2^{j-k}\sum_{l=j}^{n+k}\frac{1}{(l+k-j)!(n+k-l)!}G_{l+k-j}^k(1)\right.\\
&\;\;\;\;\;\;\;\left.+2^{j-k}\frac{1}{((n+k)+k-j)!(n+k-(n+k))!}G_{(n+k)+k-j}^k\right)\\
&=\frac{1}{2j!}\left(2^{j-k}\sum_{l=j}^{n+k}\frac{1}{(l+k-j)!(n+k-l)!}(-1)^{l+2k-j}G_{l+k-j}^k(k-1)\right.\\
&\;\;\;\;\;\;\;\left.+2^{j-k}\frac{1}{(n+2k-j)!}G_{n+2k-j}^k\right).
\end{align*}
Substituting this to equation \eqref{eqn15} yields
\begin{align*}
p(x)=\sum_{j=k}^{n+k}\frac{2^{j-k-1}}{j!}\left(\sum_{l=j}^{n+k}\frac{(-1)^{l-j}G_{l+k-j}^k(k-1)}{(l+k-j)!(n+k-l)!}+\frac{G_{n+2k-j}^k}{(n+2k-j)!}\right).
\end{align*}
\end{proof}

\smallskip
Using the fact that the set of Bernoulli polynomials of higher order forms a basis for the space of polynomials $\mathcal{P}_n$, the following theorem is established.
\begin{thrm}\label{thm4}
The following relation holds:
\begin{align*}
\sum_{l=k}^{n+k}\frac{G_{l}^k(x)}{l!(n+k-l)!}x^{n+k-l}&=\sum_{j=k}^{n+k}\frac{2^{j-1}}{j!}\left(\sum_{l=k+j-1}^{n+k}\frac{(-1)^{l-j+1+k}G_{l-j+1}^k(k-1)B_{j}^k(x)}{(l-j+1)!(n+k-l)!}\right.\\
&\;\;\;\;\;\;\;\left.-\frac{G_{n+k-j+1}^kB_{j}^k(x)}{(n+k-j+1)!}\right).
\end{align*}
\end{thrm}
\begin{proof}
Write $p(x)$ as linear combination of Bernoulli polynomials of higher order as follows:
\begin{equation}\label{eqn16}
p(x)=\sum_{l=k}^{n+k}a_lB_{l}^k(x).
\end{equation}
Now, using equation \eqref{eqnn1} and Lemma \ref{lem111}, 
\begin{align*}
a_j&=\frac{1}{j!}[p^{(j-1)}(1)-p^{(j-1)}(0)]\\
&=\frac{1}{j!}\left(2^{j-1}\sum_{l=k+j-1}^{n+k}\frac{G_{l-j+1}^k(1)1^{n+k-l}}{(l-j+1)!(n+k-l)!}-2^{j-1}\sum_{l=k+j-1}^{n+k}\frac{G_{l-j+1}^k(0)0^{n+k-l}}{(l-j+1)!(n+k-l)!}\right)\\
&=\frac{2^{j-1}}{j!}\left(\sum_{l=k+j-1}^{n+k}\frac{(-1)^{l-j+1+k}G_{l-j+1}^k(k-1)}{(l-j+1)!(n+k-l)!}-\frac{G_{n+k-j+1}^k}{(n+k-j+1)!(n+k-(n+k))!}\right).
\end{align*}
Hence, equation \eqref{eqn16} can be written as
\begin{align*}
p(x)&=\sum_{j=k}^{n+k}\frac{2^{j-1}}{j!}\left(\sum_{l=k+j-1}^{n+k}\frac{(-1)^{l-j+1+k}G_{l-j+1}^k(k-1)}{(l-j+1)!(n+k-l)!}-\frac{G_{n+k-j+1}^k}{(n+k-j+1)!}\right)B_{j}^k(x),
\end{align*}
which is exactly the identity in the theorem.
\end{proof}

\smallskip
The following corollary is an immediate consequence of Theorem \ref{thm4}.
\begin{corr}
The following equality holds:
\begin{align*}
\sum_{l=k}^{n+k}\frac{G_{l}^k(x)}{l!(n+k-l)!}x^{n+k-l}&=\sum_{j=k}^{n+k}\frac{2^{j-1}}{j!}\Bigl(\sum_{l=k+j-1}^{n+k}\frac{(-1)^{l-j+1+k}(l-j+1)_kE_{l-j+1-k}^k(k-1)B_{j}^k(x)}{(l-j+1)!(n+k-l)!}\Bigr.\\
&\;\;\;\;\;\;\;\Bigl.-\frac{(n+k-j+1)_kE_{n-j+1}^kB_{j}^k(x)}{(n+k-j+1)!}\Bigr).
\end{align*}
\end{corr}

\smallskip
Using the fact that the set of Euler polynomials of higher order forms a basis for the space of polynomials $\mathcal{P}_n$,  the following theorem is proved.
\begin{thrm}\label{thm4}
The following relation holds:
\begin{align*}
\sum_{l=k}^{n+k}\frac{G_{l}^k(x)}{l!(n+k-l)!}x^{n+k-l}==\sum_{j=k}^{n+k}\frac{2^{j-1}}{j!}\left(\sum_{l=k+j}^{n+k}\frac{(-1)^{l-j+k}G_{l-j}^k(k-1)}{(l-j)!(n+k-l)!}-\frac{G_{n+k-j}^k}{(n+k-j)!}\right)E_{j}^k(x)
\end{align*}
\end{thrm}
\begin{proof}
Write $p(x)$ as linear combination of Bernoulli polynomials of higher order as follows:
\begin{equation}\label{eqn16}
p(x)=\sum_{l=k}^{n+k}b_lE_{l}^k(x).
\end{equation}
Now, using equation \eqref{eqnn2} and Lemma \ref{lem111},
\begin{align*}
b_j&=b_j=\frac{1}{2j!}[p^{(j)}(1)+p^{(j)}(0)]\\
&=\frac{1}{2j!}\left(2^{j}\sum_{l=k+j}^{n+k}\frac{G_{l-j}^k(1)1^{n+k-l}}{(l-j)!(n+k-l)!}+2^{j}\sum_{l=k+j}^{n+k}\frac{G_{l-j}^k(0)0^{n+k-l}}{(l-j)!(n+k-l)!}\right)\\
&=\frac{1}{2j!}\left(2^{j}\sum_{l=k+j}^{n+k}\frac{G_{l-j}^k(1)}{(l-j)!(n+k-l)!}+2^{j}\frac{G_{n+k-j}^k}{(n+k-j)!}\right).
\end{align*}
Hence, equation \eqref{eqn16} can be written as
\begin{align*}
p(x)&=\sum_{j=k}^{n+k}\frac{2^{j-1}}{j!}\left(\sum_{l=k+j}^{n+k}\frac{(-1)^{l-j+k}G_{l-j}^k(k-1)}{(l-j)!(n+k-l)!}-\frac{G_{n+k-j}^k}{(n+k-j)!}\right)E_{j}^k(x),
\end{align*}
which is exactly the identity in the theorem.
\end{proof}

\smallskip
The following corollary is an immediate consequence of Theorem \ref{thm4}.
\begin{corr}
The following equality holds:
\begin{align*}
\sum_{l=k}^{n+k}\frac{G_{l}^k(x)}{l!(n+k-l)!}x^{n+k-l}&=\sum_{j=k}^{n+k}\frac{2^{j-1}}{j!}\Bigl(\sum_{l=k+j-1}^{n+k}\frac{(-1)^{l-j+1+k}(l-j)_kE_{l-j-k}^k(k-1)E_{j}^k(x)}{(l-j)!(n+k-l)!}\Bigr.\\
&\;\;\;\;\;\;\;\Bigl.-\frac{(n+k-j)_kE_{n-j}^kE_{j}^k(x)}{(n+k-j)!}\Bigr).
\end{align*}
\end{corr}

\section{Conclusion and Recommendation}
In this paper, the Bernoulli, Euler and Genocchi basis were used to derive some identities of Genocchi polynomials of higher order. Also, the results were presented in terms of Bernoulli and Euler polynomials since these polynomials are closely related to Genocchi polynomials. For further studies, the following are recommended by the authors: (1) To obtain some identities of Genocchi Polynomials of complex order $k$; (2) To generalize the results, parallel to this paper, for Apostol Genocchi
Polynomials and higher order Apostol Genocchi Polynomials.


\begin{thebibliography}{99}

\bibitem{ref1} M. Abramowitz and I.A. Stegun, \textit{Handbook of Mathematical Functions}, Dover, New York, 1970.

\bibitem{ref2} A. Adelberg, Higher Order Bernoulli Polynomials and Newton Polygons, \textit{Applications of Fibonacci Numbers} pp1-2.

\bibitem{ref3} S. Araci, E. Sen, and M. Acikgoz, Theorems on Genocchi polynomials of higher order arising from Genocchi basis, \textit{Taiwanese Journal of Mathematics and Mathematical Sciences}, \textbf{18}(2), 473--482.

\bibitem{ref4} S. Araci, E. Sen, and M. Acikgoz, Some new formulae for Genocchi Numbers and Polynomials involving Bernoulli and Euler polynomials, \textit{International Journal of Mathematics and Mathematical Sciences}, \textbf{2014}, Article IC 760613, 7 pages.

\bibitem{ref5} S. Araci, Novel Identities Involving Genocchi numbers and Polynomials Arising from Applications of Umbral Calculus, \textit{Applied Mathematics and Computation}, \textbf{233}(2014). 599--607.

\bibitem{ref6} A.F. Horadam, Genocchi Polynomials, \textit{The Fibonacci Quarterly}, \textbf{30}(1)(1992), 21--34.

\bibitem{ref7} S. Hu, D. Kim, and M.S. Kim, New Identities Involving Bernoulli, Euler and Genocchi Numbers,  \textit{Advances in Difference Equations},  \textit{74}(2013).

\bibitem{ref8} A. Isah and C. Phang, On Genocchi Operational Matrix of Fractional Integration for Solving Fractional Differential Equation,  \textit{AIP Conference Proceedings} 1795. 020015. 2017.

\bibitem{ref9} D.S. Kim, D.V. Dolgy, T. Kim, and S.H. Rim, Some Formula for the Product of Two Bernoulli and Euler Polynomials, Abstract and Applied Analysis. \textbf{2012}, Article ID 784307, 15 pages.

\bibitem{ref10} D.S. Kim and T. Kim, A note on Higher Order Bernoulli Polynomials, \textit{Journal of Inequalities and Applications}, \textbf{2013}.

\bibitem{ref11} M.S. Kim, A Note on Sums of Products of Bernoulli Number, \textit{Applied Mathematics Letters}, \textbf{24}(2011), 55--61.

\bibitem{ref12} D.S. Kim and T. Kim, Some Identities of Higher Order Euler Polynomials Arising from Euler Basis, \textit{Integral Transforms and Special Functions}, \textbf{24}(9)(2012).
DOI: 10.1080/10652469.2012.754756

\bibitem{ref13} C.S. Ryoo, T. Kim, B. Lee, and J. Choi, On the generalized $q$-Genocchi numbers and polynomials of higher order, \textit{Advances in Difference Equation}, \textbf{2011}, Article ID 424809, 8 pages.

\bibitem{ref14} Q. Zhou, Identities on Genocchi Polynomials and Genocchi Numbers Concerning Binomial Coefficients, \textit{International Journal of Analysis and Applications}, \textbf{14}(2)(2017), 140--146.
\end{thebibliography}
\end{document}